\documentclass[11pt,letter]{article}
\usepackage[utf8]{inputenc}
\usepackage[all]{xy}
\usepackage{amsmath}
\usepackage{amssymb}
\usepackage{mathrsfs}
\usepackage{euscript}
\usepackage{mathdesign}
\usepackage{amsthm}
\usepackage{bm}
\usepackage{comment}
\usepackage{framed}
\usepackage[margin=0.65in]{geometry}
\usepackage{blindtext}
\usepackage{color}
\usepackage{xcolor}
\usepackage{comment}
\usepackage{hyperref}
\hypersetup{
	colorlinks=true, 
	linktoc=all,     
	linkcolor=blue,  
}
\usepackage{fancyhdr}
\begin{document}
\large
\theoremstyle{plain}
\newtheorem{thm}{Theorem}[section] 

\theoremstyle{definition}
\newtheorem{defn}[thm]{Definition} 
\newtheorem{proposition}[thm]{Proposition} 
\newtheorem{theorem}[thm]{Theorem} 
\newtheorem{cor}[thm]{Corollary} 
\newtheorem{lemma}[thm]{Lemma}
\newtheorem{example}[thm]{Example} 
\newtheorem{remark}[thm]{Remark}

\renewcommand{\qedsymbol}{$\blacksquare$}
\newcommand{\norm}[1]{\left\lVert#1\right\rVert}
\let\oldle\le
\let\le\leqslant
\let\oldge\ge
\let\ge\geqslant
\let\oldemptyset\emptyset
\let\emptyset\varnothing
\let\subset\propersubset
\let\subset\subseteq
\let\epsilon\stupidepsilon
\let\epsilon\varepsilon
\newcommand{\onto}{\twoheadrightarrow}
\newcommand{\hto}{\stackrel{\mu^*}{\longrightarrow}}

\title{Norm Inequalities for the Fourier Coefficients of Some Almost Periodic Functions}
\author{Y. Boryshchak, A. Myers, and Y. Sagher \\ Florida Atlantic University, Department of Mathematical Sciences, Boca Raton, FL}

\maketitle



\begin{abstract}
	Using C. Fefferman's embedding of a charge space in a measure space allows us to apply standard interpolation theorems to prove norm inequalities for Besicovitch almost periodic functions. This yields an analogue of Paley's Inequality for the Fourier coefficients of periodic functions.
\end{abstract}

\section{Introduction}

The Fourier coefficients of a  $2\pi$-periodic function $f \in L^1(-\pi, \pi)$ are $\hat{f}(n):=\frac{1}{2\pi}\int_{-\pi}^{\pi}e^{-inx}f(x)dx$ with $n\in \mathbb{Z}$. For  $1\le q \le 2$ and $1/q+1/q'=1$ the Hausdorff-Young inequality reads (e.g. \cite{Zygmund} Chapter XII, Theorem 2.3):
\begin{equation}
\left(\sum_{n=-\infty}^{\infty}|\hat{f}(n)|^{q'}\right)^{1/q'}\le \left(\frac{1}{2\pi}\int_{-\pi}^{\pi}|f(x)|^q dx\right)^{1/q}.
\end{equation}
R. A. E. C. Paley \cite{Paley} extended the Hausdorff-Young  inequality to 
\begin{equation}
\label{eq:pap1a1}
\left(\sum_{n=1}^{\infty}[n^{1/q'}\hat{f}_n^*]^{q}\frac{1}{n}\right)^{1/q}\le C(q)\left(\int_{-\pi}^{\pi}|f(x)|^q dx\right)^{1/q}
\end{equation}
where $\{\hat{f}_n^*\}_{n\in \mathbb{N}}$ is the decreasing rearrangement of the sequence $\left\{|\hat{f}(n)|\right\}_{n\in \mathbb{N}}$.

In \cite{Bruno} A. Avantaggiati, G. Bruno and R. Iannacci proved a Hausdorff-Young inequality for almost periodic trigonometric polynomials, i.e., trigonometric polynomials of the form:
	\begin{equation}
	\label{eq:pap1}
	P(x)=\sum_{k=1}^{n} a(\eta_k; P) e^{i\eta_k x}, \qquad \eta_k\in \mathbb{R}, \quad a(\eta_k; P)\in \mathbb{C}
	\end{equation}
	and for functions that are in the completion of these polynomials -- the space of  Besicovitch almost periodic functions, $B^q_{ap}$, with the norm
	\begin{equation}
	\label{eq:pap2}
	\norm{P}_{B^q_{ap}}=\lim_{\tau\to \infty}\left(\frac{1}{2\tau}\int_{-\tau}^{\tau} |P(x)|^q dx\right)^{1/q}. 
	\end{equation}
Since
\begin{equation}\label{eq:intro1}
\langle e^{i\eta_jx}, e^{i\eta_kx} \rangle=\lim_{\tau\to \infty}\left(\frac{1}{2\tau}\int_{-\tau}^{\tau} e^{i\eta_jx}\cdot e^{-i\eta_kx}dx\right)=\begin{cases} 
1 & \eta_j=\eta_k \\
0 & \eta_j\ne \eta_k
\end{cases}
\end{equation}
the functions $\{ e^{i\eta_kx}\}_{k\in\mathbb{N}}$ form an orthonormal system in  $B^2_{ap}$. For $f\in B^1_{ap}$  the Fourier coefficients of $f$ are defined
\begin{equation}
\label{eq:pap2a}
a(\eta; f)=\lim_{\tau\to \infty}\left(\frac{1}{2\tau}\int_{-\tau}^{\tau} f(x)e^{-i\eta x} dx\right), \qquad \eta \in \mathbb{R}.
\end{equation}
If $f$ is a periodic function and $\eta\in \mathbb{Z}$ then $a(\eta; f)=\hat{f}(\eta)$.

In \cite{Bruno} the authors prove  the following extension of the Hausdorff -Young inequality: 
\begin{thm}
	\label{thm:pap1}
Let $1\le q \le 2$ and let  $f\in B^q_{ap}$.   If $\left\{a(\eta_n; f)\right\}_{n\in \mathbb{N}}$ is a sequence of Fourier coefficients of $f$ then
	\begin{equation}
	\label{eq:pap3}
	\left(\sum_{j=1}^{\infty} |a(\eta_j; f)|^{q'}\right)^{1/q'}\le \limsup_{\tau\to \infty}\left(\frac{1}{2\tau}\int_{-\tau}^{\tau} |f(s)|^q ds\right)^{1/q}.
	\end{equation}
\end{thm}

We shall prove a Paley-type theorem for   $B^q_{ap}$ functions. 
We denote the Lebesgue measure space on the real line by $(\mathbb{R}, \mathcal{L}(\mathbb{R}), \lambda)$. 
The definition of the Besicovitch $B^q_{ap}$ space as the completion of trigonometric polynomials (\ref{eq:pap1}) under the norm $\norm{\cdot }_{B^q_{ap}}$ suggests that one could look at the set function $\gamma$ on Lebesgue measurable subsets of $\mathbb{R}$ for which the following limit
\begin{equation}
\label{eq:pap4}
\gamma(S)=\lim_{\tau\to \infty}\frac{\lambda\left(\left\{S \cap (-\tau, \tau)\right\}\right)}{2\tau}
\end{equation}
exists. Using Banach limits one extends $\gamma$ to a finitely additive set function on all sets of $\mathcal{L}(\mathbb{R})$. Clearly $\gamma$ is a finitely additive set function, but since
\begin{equation}
\label{eq:pap5}
\sum_{n\in\mathbb{Z}}\gamma([n, n+1))=0<1= \gamma\left(\bigcup_{n\in\mathbb{Z}}[n, n+1)\right)
\end{equation}
$\gamma$ is not $\sigma$-additive.  To handle the lack of $\sigma$-additivity   we use  C. Fefferman's \cite{Fefferman} construction of a measure space $(\mathbb{R}', \mathcal{L}'(\mathbb{R}'), \gamma')$ with respect to which the $L^q(\mathbb{R}', \mathcal{L}'(\mathbb{R}'), \gamma')$ norm is equivalent to $\norm{\cdot }_{B^q_{ap}}$. This enables us to use  standard interpolation theorems to prove the following  theorem.

\begin{theorem} 
	\label{thm:pap2}
	Let $1\le q \le 2$ and let  $f\in B^q_{ap}$.   If $\left\{a(\eta_n; f)\right\}_{n\in \mathbb{N}}$ is a sequence of Fourier coefficients of $f$ then
	\begin{equation}
	\label{eq:pap1a}
	\left(\sum_{n=1}^{\infty}[n^{1/q'}a^*(\eta_n; f)]^{q}\frac{1}{n}\right)^{1/q}\le C(q)\left(\int_{\mathbb{R}}|f(x)|^q d\gamma(x)\right)^{1/q}
	\end{equation}
	where $\{a^{*}(\eta_n; f)\}_{n\in \mathbb{N}}$ is the decreasing rearrangement of the sequence $\left\{|a(\eta_n; f)|\right\}_{n\in \mathbb{N}}$.

\end{theorem}

Since the left hand side of (\ref{eq:pap1a}) is the $l^{q',q}$ -norm of the sequence $\{a(\eta_n; f)\}_{n\in\mathbb{N}}$ and for $1<q<2$ the space $l^{q',q}$ is a subspace of $l^{q'}$, except for $C(q)$, inequality (\ref{eq:pap1a}) is stronger than (\ref{eq:pap3}). Also, we shall extend (\ref{eq:pap1a}) to the case $l^{q,r}$ with $0<r < \infty$ and $1<q<2$. \\

The paper is organized as follows. Section \ref{pp:sec:charges} contains a brief theory of finitely additive measures, ``charges". From this, some facts on integration and new results on the well-definedness of $ L^{p,q} $ for charge spaces, using Cauchy sequences, are presented. In Section \ref{pp:sec:embedfeff}, we prove that the C. Fefferman construction of a measure space that contains an isometric isomorphic image of a given charge space also preserves $L^{p,q}$ norms. The set function, $ \gamma $, in (\ref{eq:pap4}) is extended to a charge in Section \ref{pp:sec:apf} using Banach Limits. Section \ref{pp:sec:paley} combines the results of the previous sections to prove a Paley-type theorem for Besicovitch almost periodic functions.

\section{Just Enough Theory of Charges}\label{pp:sec:charges}
We present a short outline of  part of the theory of charges that is used in the notes and  refer the reader to  K. Rao and  M. Rao \cite{Rao} for an excellent exposition of the subject. Some proofs of known theorems are not included here, but may be found in \cite{Yarema}.

\subsection{Charges, simple functions, and spaces of measurable functions}

Let $\Omega$ be a set and let $\mathcal{F}$ be a collection of subsets of $\Omega$. $\mathcal{F}$ is said to be a \textbf{field} on $\Omega$ if $ \Omega \in \mathcal{F} $ and $ \mathcal{F} $ is closed under finite unions and set differences. A map $ \mu:\mathcal{F} \rightarrow [0,\infty] $ is said to be a \textbf{charge} on $ \mathcal{F} $ if $ \mu(\emptyset) = 0 $  and $ \mu $ is finitely additive. A \textbf{probability charge} is a charge for which $ \mu(\Omega) = 1 $. The triple $ (\Omega,\mathcal{F},\mu) $ is called a \textbf{charge space}. \\ \\
For a charge space, $(\Omega, \mathcal{F}, \mu)$, define the \textbf{outer charge} of $E \subseteq \Omega$ by $ \bm{\mu^*(E)}=\inf\{\mu(F) : E\subset F, \, F\in \mathcal{F}\} $. Note that if $E\in \mathcal{F}$, then  $\mu^*(E)=\mu(E)$. From the finite additivity of $ \mu$, $\mu^*$ is finitely, not necessarily $\sigma$, sub-additive and so the  Carath\'{e}odory charge induced by $\mu^*$ is finitely additive \cite{Yeh}.

\begin{defn}
	Let $(\Omega, \mathcal{F}, \mu)$ be a charge space. A sequence $\{f_n\}$ of real valued functions on $\Omega$ is said to be \textbf{Cauchy in } $\bm{\mu^*}$, if  $\forall \delta>0$,
	\begin{equation}
	\lim_{m,n\to \infty} \mu^*\left(\left\{w\in \Omega : |f_m(w)-f_n(w)|>\delta\right\}\right)=0.
	\end{equation}
	$\{f_n\}$ is said to \textbf{converge in} $\bm{\mu^*}$ to $f : \Omega \to \mathbb{R}$  and write $\bm{f_n\stackrel{\mu^*}{\longrightarrow}f}$, if  $\forall \delta>0$,
	\begin{equation}
	\lim_{n\to \infty} \mu^*\left(\left\{w\in \Omega : |f(w)-f_n(w)|>\delta\right\}\right)=0.
	\end{equation}
	In other words, $f_n$ converge to $f$ in outer charge\footnote{Rao and Rao \cite{Rao} name convergence in $\mu^*$, "hazy convergence".}.
\end{defn}

\begin{defn}
	Let $(\Omega, \mathcal{F}, \mu)$ be a  charge space. We say that  an extended real valued function $f$ on $\Omega$ is a $\bm{\mu}$\textbf{-null function}, or just null function, if for every $\epsilon>0$
	\begin{equation}
	\mu^*(\{|f|>\epsilon\})=0.
	\end{equation}
	We denote the set of null functions on  $(\Omega, \mathcal{F}, \mu)$ by $\bm{\mathcal{N}(\Omega, \mathcal{F}, \mu)}$. We say that two real valued functions $f$ and $g$ on $\Omega$ are \textbf{equal $\bm{\mu}$-almost everywhere} and write $\bm{f=g}$ \textbf{$\bm{\mu}$-a.e.} (or just $f=g$ a.e. ) if $f-g$ is a $\mu$-null function. We say that $\bm{f\le g}$ \textbf{$\bm{\mu}$-a.e.} if there exists a null function, $h$, s.t. $f\le g+h$.
\end{defn}

\begin{remark}
	$f, g\in \mathcal{N}(\Omega, \mathcal{F}, \mu)$ implies that $f+g\in \mathcal{N}(\Omega, \mathcal{F}, \mu)$ and the relation $f=g$ $\mu$-a.e. is an equivalence relation.
\end{remark}

\begin{defn}
	Let $(\Omega, \mathcal{F}, \mu)$ be a charge space and let $f$ be a real valued function  on $\Omega$. Define 
	\begin{equation}
	\bm{\norm{f}_{\mathcal{F}(\Omega)}}=\min\left( 1, \inf_{a>0}\left\{\max\left[a, \mu^*(\{|f|>a\})\right]\right\}\right).
	\end{equation}
\end{defn}

Note that $\norm{\cdot}_{\mathcal{F}(\Omega)}$ is a semi-metric, that is to say there are non-zero functions, $f$, such that $\norm{f}_{\mathcal{F}(\Omega)}=0$. However, the quotient space of functions modulo the equivalence
\begin{equation}
f\sim g \quad \Leftrightarrow \quad \norm{f-g}_{\mathcal{F}(\Omega)}=0
\end{equation}
is a metric space. It is easy to show that convergence in  $\norm{\cdot}_{\mathcal{F}(\Omega)}$ is  convergence in outer charge. Similarly $\{f_n\}$ is Cauchy in $\norm{\cdot}_{\mathcal{F}(\Omega)}$ if and only if it is Cauchy in outer charge. Thus we consider the following:

\begin{defn}
	We denote the set of equivalence  classes of functions on $\Omega$ modulo $\mathcal{N}(\Omega, \mathcal{F}, \mu)$  by $\bm{F(\Omega, \mathcal{F}, \mu, \mathbb{R})}$ or just by $\bm{F(\Omega)}$. We denote the equivalence class of $ f$ by $[f]$. 
\end{defn}

	Note that $(F(\Omega),  \norm{\cdot}_{\mathcal{F}(\Omega)})$ is a metric space, with 
	
	\begin{equation}
	\norm{[f]}_{\mathcal{F}(\Omega)}=\norm{f}_{\mathcal{F}(\Omega)}.
	\end{equation}

We will not distinguish between sequences of functions which belong to the same equivalence class. We will also not distinguish between equivalence classes and their representatives.

\begin{defn}
	Let $(\Omega, \mathcal{F}, \mu)$ be a  charge space. We say that $f$ is a \textbf{simple} function if for some $n\in \mathbb{N}$, disjoint $E_j\in \mathcal{F}$ and some $\beta_j\in\mathbb{R}$, $f=\sum_{j=1}^{n}\beta_j I_{E_j}$. Since $\mathcal{F}$ is closed under finite unions one may assume that $\beta_j$ are distinct and write $f$ in the canonical form
	\begin{equation}
	f=\sum_{j=1}^{n}\beta_j I_{\{f=\beta_j\}}.
	\end{equation}
\end{defn}

Applying the standard completion of a metric space to the space of simple functions on a charge space  $(\Omega, \mathcal{F}, \mu)$ we obtain a complete metric space $\dot{T}M(\Omega)$, see \cite{Fefferman}. 

\begin{defn}

	We define an equivalence relation on sequences of simple functions  that are Cauchy in $\mu^*$, and say that $\{f_n\}$ and $\{g_n\}$  belong to the same equivalence class if
	\begin{equation}
	\label{eq:TMdot}
	\lim_{n\to \infty}\norm{f_n-g_n}_{\mathcal{F}(\Omega)}=0.
	\end{equation}
	 We denote the space of all equivalence classes of sequences of simple functions that are Cauchy  in $\mu^*$ by $\bm{\dot{T}M(\Omega, \mathcal{F}, \mu)}$.
\end{defn} 

\begin{defn}
	If $\{f_n\}\in \dot{T}M(\Omega)$, we define
	\begin{equation}
	\bm{\norm{\{f_n\}}_{\dot{T}M(\Omega)}}=\lim_{n\to \infty}\norm{f_n}_{\mathcal{F}(\Omega)}.
	\end{equation}
	We denote the null element of $\dot{T}M(\Omega, \mathcal{F}, \mu)$ by $\{0\}$.
\end{defn}

\begin{defn}\label{charge:def:order}
	Let  $\{f_n\}, \{g_n\} \in \dot{T}M(\Omega, \mathcal{F}, \mu)$. We say that $\{f_n\} \ge \{0\}$ if $\{f_n^-\}=\{0\}$. If $\{f_n-g_n\}\ge \{0\}$ we say that $\{f_n\}\ge \{g_n\}$.
	
\end{defn}

\begin{defn}
	We denote the closure of simple functions in the metric space $(F(\Omega), \norm{\cdot}_{\mathcal{F}(\Omega)})$ by $\bm{TM(\Omega)}$. If $f\in TM(\Omega)$ then we say that $f$ is \textbf{totally measurable}, or $ TM $-measurable\footnote{Rao and Rao \cite{Rao} refer to these functions as $ T_1 $-measurable}.

\end{defn}

	Note that $TM(\Omega, \mathcal{F}, \mu)$ is isometrically isomorphic to the subspace of $\dot{T}M(\Omega, \mathcal{F}, \mu)$ of those sequences, $ \{f_n\} $, for which there exists a function, $f$, s.t. $f_n\hto f$. 

\begin{remark}\label{remarkTM}
	Let $(\Omega, \mathcal{F}, \mu)$ be a  charge space and $\{f_n\}, \{g_n\} \in \dot{T}M(\Omega)$ 
	\begin{enumerate}
		\item If $\Psi :\mathbb{R} \to  \mathbb{R}$ is continuous then $\{\Psi(f_n)\}\} \in \dot{T}M(\Omega)$.
		\item If $f, g : \Omega \to\mathbb{R}$ we denote $\bm{f\lor g}=\max(f, g)$ and $\bm{f\land g}=\min(f, g)$. If $ \{f_n\}, \{g_n\} \in \dot{T}M(\Omega)$ then $\{f_n\lor g_n\} \in  \dot{T}M(\Omega)$ and $\{f_n\land g_n\} \in  \dot{T}M(\Omega)$.
		\item For any $c_1, c_2 \in \mathbb{R}$, $\{c_1f_n+c_2 g_n\} \in \dot{T}M(\Omega)$.
	\end{enumerate}
\end{remark}

The following example shows that   $ TM(\Omega)$ may fail to be a complete space, and since $ TM(\Omega)$ is a closed subspace of  $\mathcal{F}(\Omega)$, the space $\mathcal{F}(\Omega)$ may fail to be complete as well. 
\begin{example}
\label{ex:1}
	Let $\mathcal{F}$ be the field of finite and co-finite subsets of $\mathbb{N}$. Let $\mu(\{n\})=2^{-n}$ and $\mu(\mathbb{N}\setminus \{n\})=5-2^{-n}$. Let $f_n=I_{(0, n]}$. Since $\sum_{n=1}^{\infty}2^{-n}=1$, it is clear that $\{f_n\}$ is a Cauchy sequence in  $\mu^*$ and  that there is no $f\in TM(\Omega)$ so that $f_n\hto f$.
\end{example}

\subsection{Integration on $TM(\Omega)$ and on its completion}

\begin{defn}
	Let $(\Omega, \mathcal{F}, \mu)$ be a charge space. We say that a simple function, $ f=\sum_{j=1}^{n}\beta_j I_{\{f=\beta_j\}}$, is \textbf{integrable}  if  $\forall \, \beta_j\ne 0$  $\mu(\{f=\beta_j\})<\infty$ and define the integral of $f$ by
	\begin{equation}
	\int_{\Omega} f \, d\mu =\sum_{j=1}^{n}\beta_j\mu(\{f=\beta_j\}).
	\end{equation} 
\end{defn}

The theory of integration in \cite{Rao} has been developed for $ TM(\Omega) $. For our purposes, it is more useful to consider $ \dot{T}M(\Omega) $. The fundamental facts of integration remain valid for $ \dot{T}M(\Omega) $ and proofs of the following properties for Cauchy sequences, rather than convergent ones, can be modified from the existing proofs or can easily be supplied. In particular, linear combinations of simple integrable functions are integrable, and Chebyshev's inequality for simple functions also holds in charge spaces.

\begin{defn}
	Let $(\Omega, \mathcal{F}, \mu)$ be a charge space.  If there exists a sequence, $\{f_n\}$, of integrable  simple functions on $\Omega$ s.t.
	$f_n \hto f$  and 
	\begin{equation}
	\label{eq:CauchyL1}
	\lim_{m,n\to \infty}\int |f_n-f_m|d\mu=0.
	\end{equation}
	then $f$ is said to be \textbf{integrable} and 
	\begin{equation}
	\int f d\mu =\lim_{n\to \infty} \int f_n d\mu.
	\end{equation}
	is called the integral of $f$. 	The sequence $\{f_n\}$ is called a \textbf{determining sequence} for $f$.
	
	If $\{f_n\}\in \dot{T}M(\Omega)$, each $f_n$ is integrable and (\ref{eq:CauchyL1}) holds, then we say that $\{f_n\}$ is integrable and define 
	\begin{equation}
	\int_{\Omega}\{f_n\} d\mu =\lim_{n\to \infty} \int f_n d\mu.
	\end{equation}

\end{defn}

Rao and Rao (\cite{Rao}, Proposition 4.4.10) prove that the integral of a $TM$-measurable function is independent of the determining sequence. That the integral of $\{f_n\}\in \dot{T}M(\Omega)$ is independent of the representative from its equivalence class can be proved similarly, or see \cite{Yarema} for an alternative proof.

\begin{defn}
	Let $(\Omega, \mathcal{F}, \mu)$ be a charge space and $\{f_n\} \in \dot{T}M(\Omega)$. We say that $\{f_n\}$ is integrable on a set $E\in \mathcal{F}$ if $\{f_n I_{E}\}$ is integrable and define
	\begin{equation}
	\int_{E}\{f_n\} d\mu =\lim_{n\to \infty} \int_{\Omega} \{f_n I_{E} \} d\mu .
	\end{equation}
\end{defn}

Similar to Remark \ref{remarkTM} for $\dot{T}M(\Omega)$, the integral on $ \dot{T}M(\Omega) $ is linear. Also, if $\{f_n\}$ and $\{g_n\}$ are integrable in $\dot{T}M(\Omega)$, then so are $\{f_n \vee g_n \}$ and $\{f_n \wedge g_n \}$.

\qquad

\subsection{$L^p$ and $L^{p,q}$ spaces of $TM(\Omega)$ and $\dot{T}M(\Omega)$ functions}

\begin{defn}[\cite{Fefferman}]
	 Let $(\Omega, \mathcal{F}, \mu)$ be a charge space and $\{f_n\} \in \dot{T}M(\Omega)$. For $0< p<\infty$, we say that  $\{f_n\}\in \bm{\dot{L}^p(\Omega, \mathcal{F}, \mu)}$ if each $f_n$ is integrable and
	\begin{equation}
	\label{eq:cauchyinLp}
	\lim_{m,n \to \infty}\int_{\Omega}\left|(|f_n|^p-|f_m|^p)\right| \, d\mu =0.
	\end{equation}
\end{defn}
Since
\begin{equation}
\lim_{m,n \to \infty}\left|\int_{\Omega}(|f_n|^p-|f_m|^p )\, d\mu\right|\leq \lim_{m,n \to \infty}\int_{\Omega}\left|(|f_n|^p-|f_m|^p)\right| \, d\mu =0.
\end{equation}
$\lim_{n \to \infty}\int_{\Omega}|f_n|^p\, d\mu$ exists.
\begin{defn}
	If $\{f_n\}\in \dot{L}^p(\Omega, \mathcal{F}, \mu)$ we define
	\begin{equation}
	\norm{\{f_n\}}_{\dot{L^p}(\Omega)}=\lim_{n\to \infty} \norm{f_n}_{L^p(\Omega)}.
	\end{equation}
\end{defn}

\begin{defn}
	Let $(\Omega, \mathcal{F}, \mu)$ be a charge space and $\{f_n\} \in \dot{T}M(\Omega)$. We say that $\{f_n\} \in \bm{\dot{L}^{\infty}}$ if there exists $ 0 < C < \infty $ s.t.
	\begin{equation}
	\lim_{n\to \infty} \mu^*\left(\{|f_n|>C\}\right)=0.
	\end{equation}
	If $\{f_n\}\in \dot{L}^{\infty}$ we define
	\begin{equation}
	\norm{\{f_n\}}_{\dot{L^{\infty}}(\Omega)}=\inf \{C>0 : \,  \lim_{n\to \infty} \mu^*\left(\{|f_n|>C\}\right)=0\}.
	\end{equation}
\end{defn}

\begin{defn}	
Define $\bm{L^p(\Omega, \mathcal{F}, \mu)}$ to be the space of all $TM$-measurable functions, $f$, s.t. $|f|^p$ is integrable and 
	\begin{equation}
	\label{eq:erifj}
	\bm{\norm{f}_{L^p(\Omega)}}=\left(\int_{\Omega} |f|^p d\mu \right)^{1/p}<\infty
	\end{equation} 
	We define $\bm{L^{\infty}(\Omega, \mathcal{F}, \mu)}$ to be the space of $TM$-measurable functions that differ from a bounded function by a null function and define the $L^{\infty}$ norm of a $TM$-measurable function, $f$, by
	\begin{equation}
	\bm{\norm{f}_{L^{\infty}(\Omega)}}=\inf\{C >0 : \, \mu^{*}(\{|f|>C\})=0 \}.
	\end{equation}
\end{defn}

Note that for $0 < p \leq \infty$, $L^p(\Omega, \mathcal{F}, \mu)$ is isometrically isomorphic to the subspace of $\dot{L}^p(\Omega, \mathcal{F}, \mu)$ of those sequences, $ \{f_n\} $, for which there exists a function, $f$, s.t. $f_n\hto f$.

\begin{defn}
	Let $(\Omega, \mathcal{F}, \mu)$ be a  charge space and $f : \Omega\to \mathbb{R}$. For $t>0$ let 
	\begin{equation}
	\mu_f(t)=\mu^{*}(\{|f|>t\})
	\end{equation}
	The function $\mu_f :\mathbb{R}^+ \to \mathbb{R}^+$ is called the \textbf{distribution of} $\bm{f}$ on $(\Omega, \mathcal{F}, \mu)$.
\end{defn}

Clearly, $\mu_f=\mu_{|f|}$.

\begin{thm}
	\label{lemma:3}
	Let $f$ and $g$ be simple integrable functions on a charge space $(\Omega, \mathcal{F}, \mu)$. Then:
	\begin{enumerate}
		
		\item The distribution $\mu_f$ is a decreasing integrable function and 
		\begin{equation}
		\int_{\Omega}|f|d\mu=\int_0^{\infty} \mu_{f}(t)dt;
		\end{equation}

		\item If $f, g \ge 0$ or $f, g \le 0$ then
		\begin{equation}
		\int_{\Omega}\left|f-g\right|d\mu=\int_0^{\infty} |\mu_{f\lor g}(t)-\mu_{f\land g}(t)|dt.
		\end{equation}
		
		\item 
		\begin{equation}
		\int_0^{\infty} |\mu_{f}(t)-\mu_{g}(t)|dt\le \int_{\Omega}|(|f|-|g|)|d\mu.
		\end{equation}
	\end{enumerate}
\end{thm}

\begin{proof}
$ $\newline

1. Let $f=\sum_{k=1}^{N}c_k I_{\{f=c_k\}}$, with $0=c_0 < |c_1| < ... < |c_N|$. For $k=1,...,N$ let $\alpha_k=|c_k|-|c_{k-1}|$ and $E_k=\{|f|\ge |c_k|\}$. Then
\begin{equation}
|f|=\sum_{k=1}^{N}\alpha_k I_{E_k} \quad \text{with }  \, \alpha_k>0, \, E_{k+1}\subset E_k, \, E_k\in \Omega
\end{equation}
and so, applying  summation by parts,
\begin{equation}
\int_{\Omega}|f|d\mu=\sum_{k=1}^{N}|c_k| \mu(\{f=c_k\})=\sum_{k=1}^{N}\alpha_k \mu(E_k)=\int_0^{\infty} \mu_{f}(t)dt.
\end{equation}

2.  It is enough to consider the case $f, g \ge 0$. Observe that for any $w\in \Omega$, $f(w)\land g(w)$ is one of $f(w)$ and $g(w)$, and   $f(w)\lor g(w)$ is the other. Since $f, g\ge 0$, by Part 1,
\begin{equation}
\begin{aligned}
\int_0^{\infty} |\mu_{f\lor g}(t)-\mu_{f\land g}(t)|dt &=\int_0^{\infty} [\mu_{f\lor g}(t)-\mu_{f\land g}(t)]dt\\&=\int_{\Omega}\left(f\lor g-f\land g\right) d\mu\\&=\int_{\Omega}\left|f\lor g-f\land g\right| d\mu\\&=\int_{\Omega}|f-g| d\mu.
\end{aligned}
\end{equation}   

3. Since  $|f|\land |g|  \le |f|, |g|\le |f|\lor |g|$, it follows that $\mu_{|f| \land |g|}\le \mu_{|f|}, \mu_{|g|}\le \mu_{|f|\lor |g|}$ and  so by part 2,
\begin{equation}
\begin{aligned}
\int_0^{\infty} |\mu_{f}(t)-\mu_{g}(t)|dt&=\int_0^{\infty} |\mu_{|f|}(t)-\mu_{|g|}(t)|dt\\&\le\int_0^{\infty} [\mu_{|f|\lor |g|}(t)-\mu_{|f|\land |g|}(t)]dt =\int_{\Omega}|(|f|-|g|)|d\mu.
\end{aligned}
\end{equation}

\end{proof}

\begin{cor}\label{lemma:3cor}
    Let $f$ and $g$ be simple integrable functions on a charge space $(\Omega, \mathcal{F}, \mu)$. Then for any $\delta \ge 0$,
    \begin{equation}
	\begin{aligned}
	\int_{\delta}^{\infty} |\mu_{f_n}(t)-\mu_{g_n}(t)|dt\le \int_{\{|f_n|>\delta\}\cup\{|g_n|>\delta\}}|(|f_n|-|g_n|)|d\mu
	\end{aligned}
	\end{equation}
\end{cor}

\begin{proof}
    \begin{equation}
	\begin{aligned}
	\int_{\delta}^{\infty} |\mu_{f_n}(t)-\mu_{g_n}(t)|dt &= \int_{0}^{\infty} |\mu_{f_n}(t+\delta)-\mu_{g_n}(t+\delta)|dt \\&\le \int_{0}^{\infty} |\mu_{(|f_n|-\delta)\lor 0}(t)-\mu_{(|g_n|-\delta)\lor 0}(t)|dt \\ &\le \int_{\Omega}|(|f_n|-\delta)\lor 0-(|g_n|-\delta)\lor 0|d\mu \\&= \int_{\Omega}|(|f_n|\lor\delta)-\delta-(|g_n|\lor\delta-\delta)|d\mu\\&\le \int_{\{|f_n|>\delta\}\cup\{|g_n|>\delta\}}|(|f_n|-|g_n|)|d\mu
	\end{aligned}
	\end{equation}

\end{proof}

We will work with the $L^{p,q}$ space  over a charge space. Let us first review this concept on a measure space.
\begin{defn}
	Let $(\Omega, \Sigma, \mu)$ be a  measure space and let $f : \Omega\to \mathbb{R}$ be a measurable function. The \textbf{decreasing rearrangement} of $f$, is defined:
	\begin{equation}
	\bm{f^*(t)}=\inf\{\gamma: \mu_f(\gamma)\le t\}, \qquad t>0
	\end{equation}
	
\end{defn}

\begin{defn}
	Let $(\Omega, \Sigma, \mu)$ be a  measure space. For $0<p, q< \infty$ and a measurable function $f : \Omega\to \mathbb{R}$ let
	\begin{equation}
	\norm{f}_{L^{p,q}(\Omega)}=\left(\int_{0}^{\infty} \left[t^{1/p}f^*(t)\right]^q \frac{dt}{t}\right)^{1/q}
	\end{equation}
	and for $0<p \le \infty$
	\begin{equation}
	\norm{f}_{L^{p,\infty}(\Omega)}= \lim_{q\to \infty}\norm{f}_{L^{p,q}(\Omega)}=\sup \left[t^{1/p}f^*(t)\right].
	\end{equation}
	The space of all measurable functions, $f$, such that $\norm{f}_{L^{p,q}(\Omega)}<\infty$ is  denoted by $\bm{L^{p,q}(\Omega, \Sigma, \mu)}$.
\end{defn}

 One can also define $\norm{f}_{L^{p,q}(\Omega)}$ in terms of  $\mu_f$ rather than $f^*$.  This definition is convenient for functions on charge spaces.
\begin{thm}[\cite{Grafakos}, Proposition 1.4.9]
	\label{thm: Lpq for measures}
	Let $(\Omega, \Sigma, \mu)$ be a  measure space. If $0<p < \infty, \; 0 < q \leq \infty$ and $f \in L^{p,q}(\Omega, \Sigma, \mu)$ then
	\begin{equation}
	\label{eq:pap2308}
	\norm{f}_{L^{p,q}(\Omega)}=\left(p\int_{0}^{\infty}[\mu_f(s)]^{\frac{q}{p}}s^{q-1}ds\right)^{1/q}=\left(\frac{p}{q}\int_{0}^{\infty}[\mu_{|f|^q}(t)]^{\frac{q}{p}} dt\right)^{1/q}.
	\end{equation}
	and
	\begin{equation}
	    \norm{f}_{L^{p,\infty}(\Omega)}=\lim_{q\to \infty}\norm{f}_{L^{p,q}(\Omega)}=\sup_{t>0}\{ t \mu_{f}^{\frac{1}{p}}(t)\}
	\end{equation}
	
\end{thm}

Motivated by Theorem \ref{thm: Lpq for measures} we make the following definitions for functions on charge spaces.
\begin{defn}
	Let $f$ be a simple integrable function on a charge space $(\Omega, \mathcal{F}, \mu)$. For $ 0<p, q < \infty$ let 
	\begin{equation}
	\label{eq:pap2308a}
	\bm{\norm{f}_{L^{p,q}(\Omega)}}=\left(p\int_{0}^{\infty}[\mu_f(s)]^{\frac{q}{p}}s^{q-1}ds\right)^{1/q}.
	\end{equation}
	For $0<p<\infty$,
	\begin{equation}
	\bm{\norm{f}_{L^{p,\infty}(\Omega)}}=\sup_{t>0}\{ t \mu_{f}^{\frac{1}{p}}(t)\}.
	\end{equation}
\end{defn}

Since for $q < \infty $,
	\begin{equation}
	\label{eqdistdef}
	\mu_{f}(t)=\mu_{|f|^q}(t^q),
	\end{equation}
it follows that for any $0\le a< b\le \infty$
\begin{equation}
\int_{a}^{b}[\mu_f(s)]^{\frac{q}{p}}s^{q-1}ds=\frac{1}{q}\int_{a^q}^{b^q}[\mu_{|f|^q}(t)]^{\frac{q}{p}} dt,
\end{equation}
and so one may define $\norm{f}_{L^{p,q}(\Omega)}$  in terms of  $\mu_{|f|^q}$ instead of $\mu_f$.
	
\begin{thm}
\label{thm:LpqEquiv}
	Let $(\Omega, \mathcal{F}, \mu)$ be a charge space and $0<p, q<\infty$.  Let $f, g$ be  simple integrable functions.
	Then 
	\begin{equation}
	\label{eq:pap17a}
	\int_0^{\infty}|\mu^{q/p}_{f}(t)-\mu^{q/p}_{g}(t)|t^q \frac{dt}{t} = \frac{1}{q}\int_{0}^{\infty} |\mu^{q/p}_{|f|^q}(t)-\mu^{q/p}_{|g|^q}(t)|dt.
	\end{equation}

\end{thm}

\begin{proof}
	There exist $0=a_1<a_2<\dots <a_{n-1} < a_n=\infty$ so that on each interval $(a_j, a_{j+1})$ the function $\mu^{q/p}_{f}-\mu^{q/p}_{g}$ does not change signs. Then
	\begin{equation}
	\begin{aligned}
	\int_{0}^{\infty} |\mu^{q/p}_{f}(t)-\mu^{q/p}_{g}(t)|t^q \frac{dt}{t}&=\sum_{j=1}^{n-1}\int_{a_j}^{a_{j+1}} |\mu^{q/p}_{f}(t)-\mu^{q/p}_{g}(t)|t^q \frac{dt}{t}\\&=\sum_{j=1}^{n-1}\left|\int_{a_j}^{a_{j+1}} [\mu^{q/p}_{f}(t)-\mu^{q/p}_{g}(t)]t^q \frac{dt}{t}\right|\\&=\frac{1}{q}\sum_{j=1}^{n-1}\left|\int_{a^q_j}^{a^q_{j+1}} [\mu^{q/p}_{|f|^q}(s)-\mu^{q/p}_{|g|^q}(s)] ds\right|\\&=\frac{1}{q}\int_{0}^{\infty} |\mu^{q/p}_{|f|^q}(s)-\mu^{q/p}_{|g|^q}(s)| ds.
	\end{aligned}
	\end{equation}
	
\end{proof}

\begin{lemma}
	\label{lemma: fth3}
	Let  $(\Omega, \mathcal{F}, \mu)$ be a charge space. Let $\{f_n\},\{g_n\} \in \dot{T}M(\Omega)$ be such that $\forall \epsilon>0$,
	\begin{equation}
	\label{eq:fth2b}
	\lim_{n\to \infty}\mu(\{|f_n-g_n|> \epsilon\})=0.
	\end{equation}
	If there exist $\delta\ge 0$ and $ C_{\delta}, M >0$, such that $\forall n\in \mathbb{N}$,
	\begin{equation}
	\label{eq:fth1a}
	|f_n|\le M, \quad \mu_{f_n}(\delta)\le C_{\delta}, \quad \text{and} \quad 	|g_n|\le M, \quad \mu_{g_n}(\delta)\le C_{\delta}.
	\end{equation}
	then for any  $0<r<\infty$,
	\begin{equation}
	\label{eq:fth3}
	\lim_{n\to \infty}\int_{\delta}^{\infty} |\mu_{f_n}^{r}(t)-\mu_{g_n}^{r}(t)|dt=0.
	\end{equation}
	
\end{lemma}

\begin{proof}
	By (\ref{eq:fth1a}),
	\begin{equation}
	\int_{\delta}^{\infty}|\mu_{f_n}^r(t)|dt\le C_{\delta}^r M \quad \text{and} \quad \int_{\delta}^{\infty}|\mu_{g_n}^r(t)|dt\le C_{\delta}^r M.
	\end{equation}
	For $\epsilon>0$ let $A^{\epsilon}_n=\{|\mu^r_{f_n}- \mu^{r}_{g_n}|>\epsilon\}$. Then 
	\begin{equation}
	\begin{aligned}
	\int_{\delta}^{\infty}|\mu_{f_n}^{r}(t)-\mu_{g_n}^{r}(t)|dt&=\int_{\delta}^{M}|\mu_{f_n}^{r}(t)-\mu_{g_n}^{r}(t)|dt\\&=\int_{[\delta,M]\cap A^{\epsilon}_n}|\mu_{f_n}^{r}(t)-\mu_{g_n}^{r}(t)|dt+\int_{[\delta,M]\cap (A^{\epsilon}_n)^c}|\mu_{f_n}^{r}(t)-\mu_{g_n}^{r}(t)|dt\\&\le \int_{[\delta,M]\cap A^{\epsilon}_n}|\mu_{f_n}^{r}(t)-\mu_{g_n}^{r}(t)|dt+\epsilon M.
	\end{aligned}
	\end{equation}
	Let us see that
	\begin{equation}
	\lim_{n\to \infty}\int_{[\delta,M]\cap A^{\epsilon}_n}|\mu_{f_n}^{r}(t)-\mu_{g_n}^{r}(t)|dt=0.
	\end{equation}
	For any $\epsilon>0$ there exists $0<\alpha<\epsilon$ so that  $|x-y|\le \alpha$ implies $|x^r-y^r|<\epsilon$. If $t\in A^{\epsilon}_n$ then $|\mu^r_{f_n}(t)-\mu^r_{g_n}(t)|>\epsilon$ and therefore $|\mu_{f_n}(t)-\mu_{g_n}(t)|>\alpha$.
	 
	By (\ref{eq:fth2b}), for any $\epsilon_1>0$ there exists $N\in \mathbb{N}$ so that for all $n\ge N$
	\begin{equation}
	\begin{aligned}
	\int_{\{|f_n|>\delta\}\cup\{|g_n|>\delta\}}|(|f_n|-|g_n|)|d\mu&=\int_{\{|f_n-g_n|>\epsilon_1\}\cap(\{|f_n|>\delta\}\cup\{|g_n|>\delta\})}|(|f_n|-|g_n|)|d\mu\\&+\int_{\{|f_n-g_n|\le \epsilon_1\}\cap(\{|f_n|>\delta\}\cup\{|g_n|>\delta\})}|(|f_n|-|g_n|)|d\mu\\&\le M\epsilon_1 + 2C_{\delta}\epsilon_1.
	\end{aligned}
	\end{equation}
	It follows that 
	\begin{equation}
	\lim_{n\to \infty}\int_{\{|f_n|>\delta\}\cup\{|g_n|>\delta\}}|(|f_n|-|g_n|)|d\mu=0.
	\end{equation}
	
	By Chebyshev's inequality and Corollary \ref{lemma:3cor}, for any $\alpha>0$,
	\begin{equation}
	\begin{aligned}
	\alpha \lambda(\{t>\delta : \, |\mu_{f_n}(t)-\mu_{g_n}(t)|>\alpha \})&\le \int_{\delta}^{\infty} |\mu_{f_n}(t)-\mu_{g_n}(t)|dt \le \int_{\{|f_n|>\delta\}\cup\{|g_n|>\delta\}}|(|f_n|-|g_n|)|d\mu
	\end{aligned}
	\end{equation}
	proving
	\begin{equation}
	\lim_{n\to \infty}\lambda(A^{\epsilon}_{n}\cap(\delta, \infty))=0
	\end{equation}
	and by (\ref{eq:fth1a}) 
	\begin{equation}
	\lim_{n\to \infty}\int_{[\delta,M]\cap A^{\epsilon}_n}|\mu_{g_n}^{r}(t)-\mu_{f_n}^{r}(t)|dt\le \lim_{n\to \infty} C_{\delta}^r \cdot \lambda(A^{\epsilon}_{n})=0.
	\end{equation}

\end{proof}

\begin{thm}
	\label{thm:Lpq1}
	 Let $(\Omega, \mathcal{F}, \mu)$ be a  charge space and $0<r<\infty$. Let $\{f_n\}, \{g_n\} \in \dot{T}M(\Omega, \mathcal{F}, \mu)$ be  sequences of  simple integrable functions so that (\ref{eq:fth2b}) holds. 
	  If also
	\begin{equation}
	\label{eq:pap36da}
	\lim_{m,n\to \infty}\int_{0}^{\infty} |\mu^{r}_{f_n}(t)-\mu^{r}_{f_m}(t)|dt=0, \qquad \lim_{m,n\to \infty}\int_{0}^{\infty} |\mu^{r}_{g_n}(t)-\mu^{r}_{g_m}(t)|dt=0,
	\end{equation}
	then
	\begin{equation}
	\label{eq:pap36ada}
	\lim_{n\to \infty} \int_{0}^{\infty}\mu^{r}_{f_n}(t) dt=\lim_{n\to \infty} \int_{0}^{\infty}\mu^{r}_{g_n}(t) dt.
	\end{equation}
\end{thm}

\begin{proof}

	Since 
	\begin{equation}
	\label{eq:pap36fa}
	\left|\int_{0}^{\infty} \mu^{r}_{f_n}(t) dt-\int_{0}^{\infty} \mu^{r}_{g_n}(t)dt\right|\le \int_{0}^{\infty} |\mu^{r}_{f_n}(t)-\mu^{r}_{g_n}(t)|dt
	\end{equation}
	if (\ref{eq:pap36ada}) fails, then  there exist   $\epsilon>0$ and  subsequences $\{f_{n_k}\}$ and $\{g_{n_k}\}$ that satisfy (\ref{eq:fth2b}), so that for all $k\in \mathbb{N}$, 
	\begin{equation}
	\label{eq:pap37da}
	\int_{0}^{\infty} \left|\mu^{r}_{f_{n_k}}(t)-\mu^{r}_{g_{n_k}}(t)\right|dt >\varepsilon.
	\end{equation}
	To simplify the notation, we will use $\{f_{n}\}$ and $\{g_{n}\}$ instead of $\{f_{n_k}\}$ and $\{g_{n_k}\}$.
 By (\ref{eq:pap36da}), there exists $N_1\in\mathbb{N}$ such that for all $n, m \ge N_1$,
	\begin{equation}
	\label{eq:pap38a}
	\begin{aligned}
	\int_{0}^{\infty}|\mu^{r}_{f_n}(t)-\mu^{r}_{f_m}(t)|dt+\int_{0}^{\infty}|\mu^{r}_{g_m}(t)-\mu^{r}_{g_n}(t)|dt<\epsilon/12
	\end{aligned}
	\end{equation}
	
	Let $M=\max\left\{\max(|f_{N_1}|), \max(|g_{N_1}|)\right\}$ and $S=\max\left\{\mu(supp(f_{N_1})), \mu(supp(g_{N_1}))\right\}$. 
	
Since $f_{N_1}$ and $g_{N_1}$ are simple integrable functions, $M, S< \infty$ and for all $t>M$, $\mu_{f_{N_1}}(t)=\mu_{g_{N_1}}(t)=0$. By (\ref{eq:pap38a}), for all $n>{N_1}$,
	\begin{equation}
	\label{eq:pap39da1}
	\begin{aligned}
	\int_{M}^{\infty}|\mu^{r}_{f_n}(t)-\mu^{r}_{g_n}(t)|dt&\le\int_{M}^{\infty}|\mu^{r}_{f_n}(t)|dt+\int_{M}^{\infty}|\mu^{r}_{g_n}(t)|dt\\&=\int_{M}^{\infty}|\mu^{r}_{f_n}(t)-\mu^{r}_{f_{N_1}}(t)|dt+\int_{M}^{\infty}|\mu^{r}_{g_n}(t)-\mu^{r}_{g_{N_1}}(t)|dt\\&\le\int_{0}^{\infty}|\mu^{r}_{f_n}(t)-\mu^{r}_{f_{N_1}}(t)|dt+\int_{0}^{\infty}|\mu^{r}_{g_n}(t)-\mu^{r}_{g_{N_1}}(t)|dt\\&< \epsilon/12.
	\end{aligned}
	\end{equation}
	
	Since  $|f_{N_1}|$ and $|g_{N_1}|$ are integrable functions there exist $\delta>0$ s.t.
	\begin{equation}
	\int_{0}^{\delta}\mu^{r}_{f_{N_1}}(t)dt<\epsilon/12 \quad \text{and} \quad \int_{0}^{\delta}\mu^{r}_{g_{N_1}}(t)dt<\epsilon/12.
	\end{equation}
	By (\ref{eq:pap36da}),  $\forall n>N_1$
	\begin{equation}
	\label{eq:pap39dab1}
	\begin{aligned}
	\int_{0}^{\delta}\mu^{r}_{f_n}(t)dt\le \int_{0}^{\delta}\mu^{r}_{f_{N_1}}(t)dt+\int_{0}^{\delta}|\mu^r_{f_n}-\mu^{r}_{f_{N_1}}(t)|dt < \epsilon/6,
	\end{aligned}
	\end{equation}
	\begin{equation}
	\label{eq:pap39dab1a}
	\begin{aligned}
	\int_{0}^{\delta}\mu^{r}_{g_n}(t)dt\le \int_{0}^{\delta}\mu^{r}_{g_{N_1}}(t)dt+\int_{0}^{\delta}|\mu^r_{g_n}-\mu^{r}_{g_{N_1}}(t)|dt < \epsilon/6
	\end{aligned}
	\end{equation}
	and therefore
	\begin{equation}
	\label{eq:pap39dab1ab}
	\int_{0}^{\delta}|\mu^{r}_{f_n}(t)-\mu^{r}_{g_n}(t)|dt<\epsilon/3.
	\end{equation}
	Let us see that 
	\begin{equation}
	\lim_{n \to \infty} \int_{\delta}^{M}|\mu^{r}_{f_n}(t)-\mu^{r}_{g_n}(t)|dt=0.
	\end{equation}
	Since $\mu^{r}_{f_n}(t)$ and $\mu^{r}_{g_n}(t)$ are decreasing functions, by Chebyshev's inequality and (\ref{eq:pap39dab1}), (\ref{eq:pap39dab1a}), $\forall n>N_1$,
	\begin{equation}
	\delta \mu^r_{f_n}(\delta)\le \int_{0}^{\delta} \mu^{r}_{f_n}(t)dt<\epsilon/6, \qquad \delta \mu^r_{g_n}(\delta)\le \int_{0}^{\delta} \mu^{r}_{g_n}(t)dt<\epsilon/6,
	\end{equation}
	 and so for $\displaystyle C_{\delta} = \left(\frac{\epsilon}{6 \delta}\right)^{1/r}$,
	\begin{equation}
	\label{eq:pap38da3}
	\sup_{n\in \mathbb{N}}\mu_{f_n}(\delta)\le C_{\delta} \quad \text{and} \quad \sup_{n\in \mathbb{N}}\mu_{g_n}(\delta)\le C_{\delta}.
	\end{equation}
	Let 
	\begin{equation}
	\tilde{f_n}(x)=|f_n|\land M \qquad \text{and} \qquad \tilde{g_n}(x)=|g_n| \land M.
	\end{equation}
	
	Observe that for all $n\in \mathbb{N}$ and all $\delta \le  t < M$,
	\begin{equation}
	\label{eq:pap2311}
	\mu_{f_n}(t)=\mu_{\tilde{f}_n}(t) \quad \text{and} \quad \mu_{g_n}(t)=\mu_{\tilde{g}_n}(t).
	\end{equation}
Since both  $|\tilde{f}_n|\le M$, $|\tilde{g}_n|\le M$, and by (\ref{eq:pap38da3}) $\mu_{\tilde{f}_n}(\delta)\le C_{\delta}$, $\mu_{\tilde{g}_n}(\delta)\le C_{\delta}$ and the  conditions  (\ref{eq:fth2b}) hold for $\{\tilde{f}_n\}$ and $\{\tilde{g}_n\}$.
	By Lemma \ref{lemma: fth3} and (\ref{eq:pap2311}),
	\begin{equation}
	\label{eq:pap2312}
	\lim_{n\to \infty}\int_{\delta}^{M}|\mu^{r}_{f_n}(t)-\mu^{r}_{g_n}(t)|dt=\lim_{n\to \infty}\int_{\delta}^{M}|\mu^{r}_{\tilde{f}_n}(t)-\mu^{r}_{\tilde{g}_n}(t)|dt=0
	\end{equation}
	and so there exists $N > N_1 $ such that for all $n, m \ge N $,
	\begin{equation}
	\label{eq:pap2313}
	\int_{\delta}^{M}|\mu^{r}_{f_n}(t)-\mu^{r}_{g_n}(t)|dt<\epsilon/3,
	\end{equation}
	and so (\ref{eq:pap2313}), (\ref{eq:pap39dab1ab}) and (\ref{eq:pap39da1}) contradict (\ref{eq:pap37da}).

\end{proof}

\begin{defn}
	\label{defn: Lpq dot2}
	Let $(\Omega, \mathcal{F}, \mu)$ be a  charge space and $0<p, q<\infty$. We say  that  $\{f_n\} \in \dot{L}^{p,q}(\Omega, \mathcal{F}, \mu)$ if $\{f_n\}\in \dot{T}M(\Omega)$ is a  sequence of simple  integrable functions such that
	\begin{equation}
	\label{eq:pap17a1}
	\lim_{m, n \to \infty}\int_0^{\infty}|\mu^{q/p}_{f_n}(t)-\mu^{q/p}_{f_m}(t)|t^q \frac{dt}{t} =0.
	\end{equation}
	We define 
	\begin{equation}
	\label{eq:pap18a}
	\bm{\norm{\{f_n\}}_{\dot{L}^{p,q}(\Omega)}}=\lim_{n\to \infty}\left(p\int_0^{\infty} \mu^{q/p}_{f_n}(t)t^q \frac{dt}{t}\right)^{1/q}
	\end{equation}
\end{defn}

From Theorems \ref{thm:LpqEquiv} and \ref{thm:Lpq1}, it follows that if $ \{f_n\} $ and $\{g_n\}$ are equivalent in $\dot{T}M(\Omega)$ and each is in $\dot{L}^{p,q}$, then $\norm{\{f_n\}}_{\dot{L}^{p,q}(\Omega)} = \norm{\{g_n\}}_{\dot{L}^{p,q}(\Omega)}$.

\begin{remark}
Equivalence of representatives in $\dot{T}M(\Omega)$ does not necessarily imply equivalence in $\dot{L}^{p,q}(\Omega)$. For example, the sequence $ \{n^{2/p}I_{[0, n^{-1}]} \}$ is a representative of $\{0\}$ in $\dot{T}M(\Omega)$, but is not a representative of $\{0\}$ in $\dot{L}^{p,q}(\Omega)$ since $\norm{\{0\}}_{\dot{L}^{p,q}(\mathbb{R})}=0$, but $\lim_{n\to \infty}\norm{f_n}_{L^{p,q}(\mathbb{R})}=\infty$. 
\end{remark}

\begin{thm}
\label{thm:Lp infiniti}
	Let $(\Omega, \mathcal{F}, \mu)$ be a  charge space and $0<p<\infty$. Let $\{f_n\}, \{g_n\} \in \dot{T}M(\Omega)$ be sequences of simple  integrable functions such that $\forall \epsilon>0$, 
	\begin{equation}
	\label{eq:Lp infiniti2}
	\lim_{n \to \infty}\mu\left(\{|f_n-g_n|>\epsilon\}\right)=0.
	\end{equation}
	If
	\begin{equation}
	\label{eq:Lp infiniti3}
	\lim_{n,m\to \infty}\left(\sup_{t>0}\{ t^p |\mu_{f_n}(t)-\mu_{f_m}(t)| \}\right)=\lim_{n,m\to \infty}\left(\sup_{t>0}\{ t^p |\mu_{g_n}(t)-\mu_{g_m}(t)| \}\right)=0
	\end{equation}
then
	\begin{equation}
	\label{eq:Lp infiniti4}
	\lim_{n\to \infty} \sup_{t>0}\{ t^p \mu_{f_n}(t)\}=\lim_{n\to \infty} \sup_{t>0}\{ t^p \mu_{g_n}(t)\}
	\end{equation}
\end{thm}

\begin{proof}
	Let $ \epsilon > 0 $. Taking subsequences, we may assume
	\begin{equation}\label{lpinf:eq:4}
	\mu_{f_n - g_n}(1/n) = \mu\left( \left\lbrace |f_n - g_n| > \frac{1}{n} \right\rbrace \right) < \frac{1}{n}, \qquad \forall n \in \mathbb{N}.
	\end{equation}	
	By $ (\ref{eq:Lp infiniti3}) $, there exists $ N_1 \in \mathbb{N} $ so that $ \forall n \geq N_1 $,
	\begin{equation}\label{lpinf:eq:5}
	\sup_{t>0} \{ t^p |\mu_{f_n}(t) - \mu_{f_{N_1}}(t)| \} < \epsilon/9.
	\end{equation}
	There also exists $ t_{N_1} > 0 $ so that
	\begin{equation}\label{lpinf:eq:6}
	\sup_{t>0} \{t^p \mu_{f_{N_1}}(t)\} \leq t^{p}_{N_1} \mu_{f_{N_1}}(t_{N_1}) + \epsilon/9.
	\end{equation}
	For any $ n \geq N_1 $,
	\begin{equation}\label{lpinf:eq:6b}
	\begin{aligned}
	t_{N_1}^p \mu_{f_n}(t_{N_1}) &\leq \sup_{t>0} \{ t^p \mu_{f_n}(t)\}
	\leq \sup_{t>0} \{ t^p |\mu_{f_n}(t) - \mu_{f_{N_1}}(t)| \} + \sup_{t>0} \{t^p \mu_{f_{N_1}}(t)\} \\
	&< 2\epsilon/9 + t_{N_1}^p\mu_{f_{N_1}}(t_{N_1})
	\end{aligned}
	\end{equation}
	which implies $ \mu_{f_n}(t_{N_1}) \leq \epsilon/t_{N_1}^p + \mu_{f_{N_1}}(t_{N_1}) $, and 
	\begin{equation}
	\begin{aligned}
	\sup_{t>0} \{ t^p \mu_{f_n}(t)\} 
	&\leq 2\epsilon/9 + t_{N_1}^p\mu_{f_{N_1}}(t_{N_1}) - t_{N_1}^p \mu_{f_n}(t_{N_1}) + t_{N_1}^p \mu_{f_n}(t_{N_1}) \\
	&\leq 2\epsilon/9 + t_{N_1}^p | \mu_{f_n}(t_{N_1}) - \mu_{f_{N_1}}(t_{N_1})| + t_{N_1}^p \mu_{f_n}(t_{N_1}) \\
	&\leq 2\epsilon/9 + \sup_{t>0} \{ t^p |\mu_{f_n}(t) - \mu_{f_{N_1}}(t)|  \} + t_{N_1}^p \mu_{f_n}(t_{N_1}) \\
	&\leq \epsilon/3 + t_{N_1}^p \mu_{f_n}(t_{N_1}).	
	\end{aligned}
	\end{equation}
	
	Similarly, there exists $ N_2 \in \mathbb{N} $ and $ t_{N_2} > 0 $ so that  $ \forall n \geq N_2 $,
	\begin{equation}
	\label{lpinf:eq:5a}
	\sup_{t>0}\{t^p \mu_{g_n}(t) \} \leq t_{N_2}^p \mu_{g_{n}}(t_{N_2}) + \epsilon/3 \qquad \text{ and } \qquad \mu_{g_n}(t_{N_2}) \leq \epsilon/t_{N_2}^p + \mu_{g_{N_2}}(t_{N_2}) .
	\end{equation}
	Set $ A = \max(\epsilon/t_{N_1}^p + \mu_{ f_{N_1}}(t_{N_1}),\epsilon/t_{N_2}^p + \mu_{g_{N_2}}(t_{N_2})) $. Let $ N > \max(N_1, N_2, 3t_{N_1}^p/\epsilon, 3t_{N_2}^p/\epsilon) $ and
	\begin{equation}\label{lpinf:eq:6d}
	\max\left( \left|t_{N_1}^p - (t_{N_1}-1/N)^p\right|, \left|t_{N_2}^p - (t_{N_2}-1/N)^p\right| \right) < \frac{\epsilon}{3+ A}
	\end{equation}
	Let $ n > N $ and without loss of generality, assume $ \sup_{t>0}\{t^p \mu_{f_n}(t)\} \geq \sup_{t>0} \{t^p\mu_{g_n}(t)\} $.
	Noting that for $ 0 < \delta < t $,
	\begin{equation}
	\mu_{f_n}(t) \leq \mu_{f_n - g_n}(\delta) + \mu_{g_n}(t-\delta),
	\end{equation}
	it follows
	\begin{equation}
	\begin{aligned}\label{lpinf:eq:6e}
	\left|\sup_{t>0} \{t^p \mu_{f_n}(t) \} -  \sup_{t>0} \{t^p \mu_{g_n}(t) \} \right| 
	&\leq t_{N_1}^p \mu_{f_{n}}(t_{N_1}) + \epsilon/3 - \left( t_{N_1} - 1/N \right)^p \mu_{g_n}(t_{N_1}-1/N) \\
	&\leq t_{N_1}^p \mu_{f_{n}}(t_{N_1}) + \epsilon/3 - \left( t_{N_1} - 1/N \right)^p \left[ \mu_{f_n}(t_{N_1}) - \mu_{f_n - g_n}(1/N) \right] \\
	&\leq \mu_{f_n}(t_{N_1}) \left|t_{N_1}^p - (t_{N_1}-1/N)^p\right| + \epsilon/3 + \epsilon/3 < \epsilon.
	\end{aligned}
	\end{equation}
	Having obtained the conclusion for subsequences of $ \{f_n\} $ and $ \{g_n\} $, the conclusion follows as the limits in $ (\ref{eq:Lp infiniti4}) $ exist. 
\end{proof}

\begin{defn}
		Let $(\Omega, \mathcal{F}, \mu)$ be a  charge space and $0<p<\infty$. We say  that  $\{f_n\} \in \dot{T}M(\Omega)$ belongs to the space $\bm{\dot{L}^{p,\infty}(\Omega, \mathcal{F}, \mu)}$ if each $f_n$ is integrable and 
	\begin{equation}
	\label{eq:pap17ab}
	\lim_{m,n \to \infty}\left(\sup_{t>0}\{ t |\mu_{f_n}^{\frac{1}{p}}(t)-\mu_{f_m}^{\frac{1}{p}}(t)| \}\right)=0.
	\end{equation}
	We  define 
	\begin{equation}
	\label{eq:pap18ab}
	\bm{\norm{\{f_n\}}_{\dot{L}^{p,\infty}(\Omega)}}=\lim_{n\to \infty}\norm{f_n}_{L^{p, \infty}(\Omega)}=\lim_{n\to \infty} \sup_{t>0}\{ t \mu_{f_n}^{\frac{1}{p}}(t)\}.
	\end{equation}
\end{defn}

\section{Embedding of a charge space into a measure space}\label{pp:sec:embedfeff}

\subsection{A Theorem of C. Fefferman  }

In \cite{Fefferman} C. Fefferman embedded  charge spaces in  measure spaces. His result is summarized in parts (1) - (3) of Theorem \ref{thm:cfs1}. The application to $L^{p, q}$ spaces necessitates a modest extension and this is part (4) of Theorem \ref{thm:cfs1}.

The notation we use in this chapter is consistent with the notation of the previous two chapters but is different from the notation used in \cite{Fefferman}.

\begin{theorem}
	\label{thm:cfs1}
	Let $(\Omega, \mathcal{F}, \mu)$ be a charge space. There is a  measure space $(\Omega', \Sigma', \mu')$ and an order-preserving, multiplication-preserving isometric isomorphism $\phi$ from $ \dot{T}M(\Omega, \mathcal{F}, \mu)$ onto $TM(\Omega', \Sigma', \mu')$, which we denote by $\mathcal{M}(\Omega', \Sigma', \mu')$,	such that the following holds:
	\begin{enumerate}
		\item If $\{f_n\}\in \dot{T}M(\Omega, \mathcal{F}, \mu)$ is an indicator function  , then $\phi(\{f_n\})\in \mathcal{M}(\Omega', \Sigma', \mu')$ is an indicator function. If $\{f_n\}\in \dot{T}M(\Omega, \mathcal{F}, \mu)$ is a simple function, then $\phi(\{f_n\})\in \mathcal{M}(\Omega', \Sigma', \mu')$ is a simple function.
		
		\item If $1\le p\le \infty$, then $\phi$ takes $\dot{L}^p(\Omega, \mathcal{F}, \mu)$ onto $L^p(\Omega', \Sigma', \mu')$ preserving the $L^p$ norm.
		
		\item If $\{f_n\}\in \dot{L}^{1} (\Omega, \mathcal{F}, \mu)$, then $\int_{\Omega}\{f_n\} d\mu= \int_{\Omega'}\phi(\{f_n\}) d\mu'$.
		
		\item If   $0< p, q < \infty$,   then $\phi$ takes $\dot{L}^{p,q}(\Omega, \mathcal{F}, \mu)$ onto $L^{p,q}(\Omega', \Sigma', \mu')$ preserving the $L^{p,q}$ norm. It also maps $\dot{L}^{p,\infty}(\Omega, \mathcal{F}, \mu)$ onto $\dot{L}^{p,\infty}(\Omega', \Sigma', \mu')$ preserving the $L^{p,\infty}$ norm.

	\end{enumerate}
\end{theorem}

We include the following background material from \cite{Fefferman} for the convenience of the reader.

	\begin{defn}
		Let $\bm{\mathcal{B}_0}$  be the subset of $\dot{T}M(\Omega, \mathcal{F}, \mu)$ of all indicator functions of sets in the field $\mathcal{F}$. Let  $\bm{\mathcal{B}}$ be the closure of $\mathcal{B}_0$ in $\dot{T}M(\Omega, \mathcal{F}, \mu)$ with respect to the metric defined in (\ref{eq:TMdot}). 

	\end{defn}

	\begin{defn}
		We define the union function $\bm{\bigcup_0}: \mathcal{B}_0\times \mathcal{B}_0 \to \mathcal{B}_0 \subseteq \mathcal{B}$, by
		\begin{equation}
		\bm{\bigcup_0(I_E, I_F)}=I_{E\cup F}
		\end{equation}
	\end{defn}		

	  Since $\bigcup_0$ is uniformly continuous on $\mathcal{B}_0\times \mathcal{B}_0$, we denote the  uniformly continuous extension of $\bigcup_0$  to a function $\bm{\bigcup} : \mathcal{B}\times \mathcal{B} \to \mathcal{B}$.

	\begin{defn}
		The function $\bm{N_0}: \mathcal{B}_0\to \mathcal{B}_0 \subseteq \mathcal{B}$ is defined by 
		\begin{equation}
		\bm{N_0(I_E)}=I_{\Omega\setminus E}.
		\end{equation}
		Since $|I_E-I_F|=|I_{(\Omega \setminus E)}-I_{(\Omega \setminus F)}|$, $N_0$ is an isometry and so extends to a uniformly continuous function $N : \mathcal{B}\to \mathcal{B}$. 
		Consistent with the notation of Boolean algebras, we denote $N(F)$ by $\sim F$. 
	\end{defn}

	\begin{defn}
		Define the intersection on $\mathcal{B} \times \mathcal{B}$ by
		\begin{equation}
		\{ I_{F_n} \} \cap \{ I_{G_n} \} =\{ I_{F_n \cap G_n} \} \qquad \{ I_{F_n} \}, \{ I_{G_n} \}\in \mathcal{B}.
		\end{equation}
	\end{defn}

	Since $\cap$ is a composition of  the complementations and unions which are uniformly continuous functions,  it is also uniformly continuous.
	
	We will now define a charge on $\mathcal{B}$. If $\{I_{E_n}\}$ and $\{I_{F_n}\}$   are Cauchy in charge and $\{I_{E_n}\} = \{I_{F_n}\}$, that is,
	\begin{equation}
	    \lim_{n\to \infty}\norm{I_{E_n}-I_{F_n}}_{\mathcal{F}(\Omega)}=0,
	\end{equation}
	it follows that
	\begin{equation}
	    \lim_{n\to \infty}\mu(E_n) = \lim_{n\to \infty}\mu(F_n)
	\end{equation}    
which justifies:
	\begin{defn}
	For $E=\{I_{E_n}\} \in \mathcal{B}$, let 
	\begin{equation}
	\bm{\mu_1(E)}=\lim_{n\to \infty}\mu(E_n).
	\end{equation} 
	
	\end{defn}

	\begin{lemma}\label{fefferman:lemma1}
		\footnote{Lemma 1 in \cite{Fefferman}. } $(\mathcal{B}, \cap, \cup, \sim)$ is a Boolean algebra and $\mu_1$ is positive and finitely additive on $\mathcal{B}$. If $F\in \mathcal{B}$ and $\mu_1(F)=0$, then $F$ is the null element of the Boolean algebra. 
	\end{lemma}

	By the Stone Representation Theorem (\cite{Dunford}), there is a compact extremally disconnected space, $\Omega'$, and a field $\mathcal{F}'$ of subsets of $\Omega'$ so that $\mathcal{F}'$ is isomorphic as a Boolean algebra with $\mathcal{B}$. The set of subsets which are both closed and open is closed under finite unions intersections and complements and therefore forms a field. Let $\bm{\phi_1}: \mathcal{B}\to \mathcal{F}'$ denote the isomorphism. Then $\phi_1$ induces a positive, finitely additive set function $\mu_1'$ on $\mathcal{F}'$ defined in the obvious way using $\phi_1$ and $\mu_1$. 
	The isomorphism proves that the conclusion of Lemma $ \ref{fefferman:lemma1} $ holds on $(\Omega', \mathcal{F}', \mu'_1)$.
	The field $\mathcal{F}'$ need not be a sigma-field, however,
	
	\begin{lemma}
		$\mu_1'$ is countably additive on $\mathcal{F}'$.
	\end{lemma}
    \begin{proof}
	Let $ A_1, A_2,\ldots
	 \in \mathcal{F}' $ be disjoint and let $ A = \bigcup_{i=1}^{\infty} A_i $. If $A \in \mathcal{F}'$ then $A$ is closed, hence compact. Since each $ A_i \in \mathcal{F}' $ is open, $ \{A_i\}_{i=1}^{\infty} $ is an open cover of $ A $, and it follows that only finitely many of the $ A_i \neq \emptyset $, proving that $ \mu_1' $ is (trivially) countably additive on $ \mathcal{F}' $.
    \end{proof}
	
	Since $\mu_1'$ is countably additive on $\mathcal{F}'$, by Carath\'{e}odory's Theorem there exists an extension of $\mu_1'$ to a positive measure $\mu'$ on $\Sigma'$, the sigma field generated by $\mathcal{F}'$.
	We shall show that $(\Omega', \Sigma', \mu')$ is the measure space of Theorem \ref{thm:cfs1}.

    We extend $\phi_1: \mathcal{B}\to \mathcal{F}'$ to simple functions, $\bm{\phi_2}: \mathcal{S}(\mathcal{B}, \mu_1)\to TM(\Omega', \mathcal{F}', \mu')$: 
    \begin{equation*}
	\phi_2\left(\sum_{i=1}^{n}\alpha_iE_i\right)=\sum_{i=1}^{n}\alpha_i\phi_1(I_{E_i}).
	\end{equation*}
	
	Since $\mathcal{S}(\Omega, \mathcal{F}, \mu)$ is dense in $\dot{T}M(\Omega)$ we can extend  $\phi_2$  continuously to an isomorphism 
	\begin{equation}
	\bm{\phi} : \dot{T}M(\Omega, \mathcal{F}, \mu) \to \dot{T}M(\Omega', \Sigma', \mu')
	\end{equation}
	such that for all $\{f_n\}\in \dot{T}M(\Omega)$,
	\begin{equation}
	\phi\left(\{f_n\right\})=\lim_{n\to \infty}\phi_2(f_n),
	\end{equation}
where the limit on the right hand side is the limit in the measure $\mu'$. It is clear that $\phi\left(\{f_n\right\})\in \dot{T}M(\Omega', \Sigma', \mu')$. By the construction,  $(\Omega', \Sigma', \mu')$ is a measure space and since in measure spaces, sequences which are Cauchy in measure converge in measure, we can identify the sequence $\{f_n\} \in \dot{T}M(\Omega', \Sigma', \mu') $ with the limit in measure of $  \{f_n\} $. In other words, $\dot{T}M(\Omega', \Sigma', \mu')= \mathcal{M}(\Omega', \Sigma', \mu')$, the space of measurable functions on $(\Omega', \Sigma', \mu') $.

	From the definition of $\phi$ it is immediate that $\phi$ is an isometric isomorphism that preserves the order as defined in Definition \ref{charge:def:order}, mapping $E\in \mathcal{B}$ to indicator functions $I_{\phi(E)}\in \mathcal{M}(\Omega', \Sigma', \mu')$. 

	The proofs of parts 2 and 3 in Theorem \ref{thm:cfs1} may also be found in the original work, \cite{Fefferman}. Closely following the proof of part 2, we proceed to prove part 4: 
		
		If $G\in \mathcal{B}$  then $G\in \dot{L}^{p,q}(\Omega, \mathcal{F}, \mu)$ and $\norm{G}^{q}_{\dot{L}^{p,q}}=\left(\frac{p}{q}\mu_1(G)\right)^{1/p}$. In this case $\phi^{-1}$ maps $\mu_1'$-integrable simple functions in ${L}^{p,q}(\Omega', \Sigma', \mu')$ to elements in $\dot{L}^{p,q}(\Omega, \mathcal{F}, \mu)$ and preserves the $L^{p,q}$ norm. Therefore $\phi^{-1}$ takes $L^{p,q}(\Omega', \Sigma', \mu')$ into $\dot{L}^{p,q}(\Omega, \mathcal{F}, \mu)$ preserving norms. But $\dot{L}^{p,q}(\Omega ', \mathcal{F}', \mu_1')=L^{p,q}(\Omega', \Sigma', \mu')$, so $\phi^{-1}$ takes $L^{p,q}(\Omega', \Sigma', \mu')$ isometrically into  $\dot{L}^{p,q}(\Omega, \mathcal{F}, \mu)$.
		
		If $E\in \mathcal{F}$ it is clear  that $I_{\phi(E)}\in L^{p,q}(\Omega', \Sigma', \mu')$. Since $I_E=\phi^{-1}(I_{\phi(E)})$, we have $I_E\in image(\phi^{-1})$. On the other hand, the linear span of $\{I_E: \mu(E) < \infty \}$ is dense in $\dot{L}^{p,q}(\Omega, \mathcal{F}, \mu)$. So $\phi^{-1}$ is onto. This completes the proof of the theorem.

\begin{remark}
	\label{rmk:msch}
	By Theorem \ref{thm:cfs1}, any property of functions on a measure space that can be formulated in terms of operations preserved by the isomorphism $ \phi $, $ \Vert \cdot \Vert_{\dot{L}^p} $, $ \Vert \cdot \Vert_{\dot{L}^{p,q}} $, and the integral of $ f $ is true for functions on a charge space. Note that one may prove many such statements without embedding a charge space into a measure space.
\end{remark}

\section{Almost Periodic Functions}\label{pp:sec:apf}

\subsection{Shift invariant finitely additive probability charge on $\mathbb{R}$ and its extension to a measure space.}
\label{sec:gamma charge}

The following construction is commonly used in Ergodic Theory, see e.g. \cite{Rao}, pp. 39-41.

Consider the set function $\gamma$ on Lebesgue measurable subsets of $\mathbb{R}$ for which the following limit exists:
\begin{equation}
\gamma(S)=\lim_{\tau\to \infty}\frac{\lambda\left(\left\{S \cap (-\tau, \tau)\right\}\right)}{2\tau}.
\end{equation}

In this section we extend $\gamma$ to a charge on $\mathcal{L}(\mathbb{R})$.

Let $L^{\infty}(\mathbb{R}, \mathcal{L}, \lambda)$ be the space of $\lambda$-essentially bounded Lebesgue measurable functions with the norm $\norm{f}_{L^{\infty}}=\text{$\lambda$-ess.}\sup_{t\in\mathbb{R}}|f(t)| $.
Let 
\begin{equation}
\bm{M}=\left\{f \in L^{\infty}(\mathbb{R}): \quad \left|\int_{a}^{b}f(s)\,ds\right|\le B(f)<\infty \text{ for all } -\infty< a \le b<\infty \right\}
\end{equation}

Then $M$ is a subspace of $L^{\infty}(\mathbb{R}, \mathcal{L}, \lambda)$. Let $\tau_h(f)(x)=f(x+h)$ for all $h\in \mathbb{R}$.

For $f\in  L^{\infty}(\mathbb{R}, \mathcal{L}, \lambda)$ let 
\begin{equation}
\bm{p(f)}=\text{$\lambda$-ess.}\sup_{t\in\mathbb{R}}f(t).
\end{equation}
Then $p$ is a sublinear functional on $L^{\infty}(\mathbb{R}, \mathcal{L}, \lambda)$ that is nonnegative on $M$.
Indeed if $p(f)=\text{$\lambda$-ess.}\sup_{t\in\mathbb{R}}f(t)=\delta<0$ then $\left|\int_{-t}^{t}f(s)\, ds\right| \ge |\delta \cdot 2t| \to\infty$ as $t\to \infty$ and therefore not bounded by any real number.

By the Hahn-Banach Theorem there exists a linear functional $T :  L^{\infty}(\mathbb{R}, \mathcal{L}, \lambda)\to \mathbb{R}$, an extension of the 0 functional on $M$,  bounded by the sublinear functional $p$. \\

$T$  satisfies the following properties.
\begin{enumerate}
	
	\item $\bm{T(f)\ge 0}$  $\forall f\ge 0$.\quad  
	
	Since for $f\ge0$,
	\begin{equation}
	-T(f)=T(-f)\le p(-f)=\text{$\lambda$-ess.}\sup_{t\in\mathbb{R}}(-f(t))\le 0.
	\end{equation}
	
	\item $\bm{T(1)=1}$. \quad 
	
	Since
	\begin{equation}
	T(1)\le p(1)=1 \qquad and \qquad -T(1)=T(-1)\le p(-1)=-1. 
	\end{equation}

	\item $T\bm{(\tau_h(f))=T(f)}$ for all $h\in \mathbb{R}$.\quad 
	
	Let $f\in  L^{\infty}(\mathbb{R}, \mathcal{L}, \lambda),\, -\infty < a<b<\infty$ and $h\ge 0$.
	If $h\le b-a$ then 
	\begin{equation*}
	\begin{aligned}
	\left|\int_a^b \left[f(s)-\tau_h(f)(s)\right]ds\right|&=\left|\int_a^b \left[f(s)-f(s+h)\right]ds\right|=\left|\int_a^b f(s) \, ds -\int_{a+h}^{b+h} f(s) \, ds\right| \\&\le \left|\int_{a}^{a+h} f(s) \, ds\right|+\left|\int_{b}^{b+h} f(s) \, ds \right| \le 2\cdot \text{$\lambda$-ess.}\sup_{t\in\mathbb{R}}f(t)\cdot h<\infty.
	\end{aligned}
	\end{equation*}
	If $h>b-a$ then
	
	\begin{equation*}
	\begin{aligned}
	\left|\int_a^b \left[f(s)-\tau_h(f)(s)\right]ds\right|&=\left|\int_a^b \left[f(s)-f(s+h)\right]ds\right|=\left|\int_a^b f(s) \, ds -\int_{a+h}^{b+h} f(s) \, ds\right| \\&\le \left|\int_{a}^{b} f(s) \, ds \right|+\left|\int_{a+h}^{b+h} f(s) \, ds \right|\le 2\cdot \text{$\lambda$-ess.}\sup_{t\in\mathbb{R}}f(t)\cdot h<\infty
	\end{aligned}
	\end{equation*}

	proving that for all $h \geq 0 $,  $\left(f-\tau_h(f)\right)\in M$. Applying the result to $\tau_{-h}(f)$ for $h \geq 0$, we have $ (\tau_{-h}(f) - f) \in M $ and so for all $h$, $(f - \tau_h(f)) \in M$. Since $T$ restricted to $M$ is zero,
	\begin{equation}
	T(f)-T\tau_h(f)=T\left( f-\tau_h(f)\right)=0.
	\end{equation}

	\item $ \bm{\lim_{|t|\to\infty}\left(\textbf{$\lambda$-ess.inf }f(t)\right)\le T(f)\le\lim_{|t| \to\infty} \left(\textbf{$\lambda$-ess.sup }f(t) \right)}$. 
	
	From	
	\begin{equation*}
	T(f)\le p(f)= \text{$\lambda$-ess.}\sup_{t\in\mathbb{R}}f(t)\quad \text{and} \quad -T(f)=T(-f)\le p(-f)=\text{$\lambda$-ess.}\sup_{t\in\mathbb{R}}(-f(t)) =-\text{$\lambda$-ess.}\inf_{t\in\mathbb{R}}f(t) 
	\end{equation*}
	it follows that
	\begin{equation}
	\label{eq:prop4}
	\text{$\lambda$-ess.}\inf_{t\in\mathbb{R}}f(t)\le T(f)\le \text{$\lambda$-ess.}\sup_{t\in\mathbb{R}}f(t)
	\end{equation}
	If  $T(f)>\lim_{|t|\to \infty}\lambda-\text{ess.}\sup_{t\in\mathbb{R}}f(t)$ then $T(f)$ is larger than the  essential supremum  of $f$ on the  interval $(-\infty, -t_0)\cup (t_0, \infty)$. But then $$T(f-fI_{[-t_0,t_0]})=T(f)-T(fI_{[-t_0,t_0]})=T(f)>\lambda-\text{ess.}\sup_{t\in\mathbb{R}}\left(f(t)I_{(-\infty, -t_0)\cup (t_0, \infty)}\right)$$
	which is a contradiction to (\ref{eq:prop4})

	\item  $\bm{|T(f)|\le \norm{f}_{L^{\infty}}}$. \\
	Since 
	\begin{equation*}
	T(f)\le p(f)= \text{$\lambda$-ess.}\sup_{t\in\mathbb{R}}f(t) \le \text{$\lambda$-ess.}\sup_{t\in\mathbb{R}}|f(t)| = \norm{f}_{L^{\infty}}
	\end{equation*}

	\item If $\lim_{|t|\to \infty}f(t)$ exists then $\bm{T(f)=\lim_{|t|\to \infty}f(t)}$. \\
	This follows  from property 4.
\end{enumerate}

Let $E$ be a Lebesgue measurable subset of $\mathbb{R}$ and
\begin{equation}
\label{eq:fsube}
\bm{\lambda_E(t)}=\frac{\lambda\left(E\cap[-t,t]\right)}{2t}.
\end{equation}
We define 
\begin{equation}
\label{eq:defgamma}
\bm{\gamma(E)}=T(\lambda_E),
\end{equation}
where $T$ is a Banach limit defined above.

To prove the additivity of $\gamma$ let $E, F \in \mathcal{L}(\mathbb{R})$ and $E\cap F=\emptyset$. Then
\begin{equation}
\lambda_{E\cup F}(t)=\frac{\lambda\left((E\cup F)\cap[-t,t]\right)}{2t}=\frac{\lambda\left(E\cap[-t,t]\right)}{2t}+\frac{\lambda\left(F\cap[-t,t]\right)}{2t}=\lambda_E(t)+\lambda_F(t)
\end{equation}
Since $T$ is linear 
\begin{equation}
\gamma(E\cup F)=T(\lambda_{E\cup F})=T(\lambda_E+\lambda_F)=\gamma(E)+\gamma(F)
\end{equation}

By Property 6 of $T$, when the following limit exists, we have
\begin{equation}
\label{eq:ex111}
\gamma(E)=\lim_{t\to \infty}\frac{\lambda\left(E\cap[-t,t]\right)}{2t}.
\end{equation}
Since 
\begin{equation}
\label{eq:ex112}
\gamma(\emptyset)=\lim_{t\to \infty}\frac{\lambda\left(\emptyset\cap[-t,t]\right)}{2t}=0
\end{equation}
and
\begin{equation}
\label{eq:ex113}
\gamma(\mathbb{R})=\lim_{t\to \infty}\frac{\lambda\left(\mathbb{R}\cap[-t,t]\right)}{2t}=1,
\end{equation}
$(\mathbb{R}, \mathcal{L}(\mathbb{R}), \gamma)$ is a probability charge space.

\begin{defn}
	We say that $E\subset \mathbb{R}$ is a $q$-periodic set if $I_E$ is a $q$-periodic function, that is, periodic with period $q$.
\end{defn}

\begin{thm}
    For any  $q$-periodic measurable set $E \subseteq \mathbb{R}$ and any $a\in \mathbb{R}$,  
    \begin{equation}
    \gamma(E)=\lim_{t\to \infty}\frac{\lambda\left(E \cap[-t,t]\right)}{2t}=\frac{\lambda\left(E\cap[a, a+q]\right)}{q}
    \end{equation}  
\end{thm}

\begin{cor}
    If $f$ is a $q$-periodic function on $\mathbb{R}$, then for any $s>0$, $ E_s = \{|f|>s \}$ is a $q$-periodic set, and therefore for any $a \in \mathbb{R}$
    \begin{equation}
        \gamma_f(s) = \lim\limits_{t \rightarrow \infty} \frac{\lambda(\{|f| > s \} \cap [-t,t] )}{2t} = \frac{\lambda( \{ |f|>s \} \cap [a,a+q] )}{q}.
    \end{equation}
\end{cor}

\begin{thm}\label{thmpperiodic}
\label{thm:periodicgamma}
    If $f$ is an integrable $q$-periodic function, then for any $a \in \mathbb{R}$, $\int_{\mathbb{R}} f d\gamma = \frac{1}{2q} \int_{a-q}^{a+q} f(x)dx $
\end{thm}

\begin{proof}
    Let $ f $ be a non-negative integrable $q$-periodic function. Since $ f $ is periodic,
	\begin{equation}
	\begin{aligned}
	\gamma_f(s) &= \lim\limits_{t \rightarrow \infty} \frac{\lambda\left( \{ f > s \} \cap [-t,t] \right)}{2t} = \frac{\lambda\left( \{ f > s \} \cap [-q,q] \right)}{2q} = \frac{1}{2q} \lambda_{fI_{[-q,q]}}(s).
	\end{aligned}
	\end{equation}
	By Lemma (\ref{lemma:3}),
	\begin{equation}\label{pperiodic}
	\int_{\mathbb{R}} f d\gamma = \int_{0}^{\infty} \gamma_{f}(s)ds = \frac{1}{2q}\int_{0}^{\infty} \lambda_{fI_{[-q,q]}}(s)ds = \frac{1}{2q} \int_{-q}^{q} f(x)dx = \frac{1}{2q}\int_{a-q}^{a+q} f(x)dx
	\end{equation}
	proving the claim for non-negative functions. More generally, if $ f $ is any $q$-periodic integrable function, we apply (\ref{pperiodic}) to $ f^+ $ and to $ f^- $. 
\end{proof}

\begin{thm}
Any almost periodic trigonometric polynomial $P_n(x)$ is integrable over $(\mathbb{R}, \mathcal{L}(\mathbb{R}), \gamma)$ with
\begin{equation}
\int_{\mathbb{R}}P_n d\gamma=\lim_{T\to \infty}\frac{1}{2T}\int_{-T}^{T} P_n(x) dx
\end{equation}
\end{thm}

\begin{proof}
    Apply Theorem \ref{thmpperiodic} to each of the terms, $e^{i\eta_k x}$.  
\end{proof}

Applying (\ref{pperiodic}) to (\ref{eq:intro1}),

 \begin{thm}
    $\left\{e^{i\eta_k x}\right\}$ with $\eta_k\in \mathbb{R}$, $k\in \mathbb{N}$ is an orthonormal system in $L^2(\mathbb{R}, \mathcal{L}(\mathbb{R}), \gamma)$.
\end{thm}

\begin{cor}[Bessel's Inequality]
If $|f|^2$ is integrable over $(\mathbb{R}, \mathcal{L}(\mathbb{R}), \gamma)$ and $\eta_k$ are all different, then 

\begin{equation}
\label{eq:m0}
\sum_{k=0}^{\infty}\left| \int_{\mathbb{R}} f(x)e^{-i\eta_k x} d\gamma \right|^2 \leq \int_{\mathbb{R}} \left|f\right|^2 d\gamma
\end{equation}

\end{cor}

\qquad

\section{A Paley type theorem for almost periodic functions}\label{pp:sec:paley}

We define the following extension of the Besicovitch $B^p_{a p}$ space. 

\begin{defn}
	For $0<p, q< \infty$ define  $ B^{p, q}_{ap}(\mathbb{R}, \mathcal{L}(\mathbb{R}), \gamma)$ to be the completion of almost periodic trigonometric polynomials under the $\norm{\cdot}_{L^{p,q}(\mathbb{R}, \mathcal{L}, \gamma)}$ norm.	
\end{defn}

By Remark $\ref{rmk:msch}$, $L^{p, p}(\mathbb{R}, \mathcal{L}(\mathbb{R}), \gamma)=L^{p}(\mathbb{R}, \mathcal{L}(\mathbb{R}), \gamma)$, and so the theory of the spaces $B^{p, q}_{ap}(\mathbb{R}, \mathcal{L}(\mathbb{R}), \gamma)$ is an extension of the theory of the Besicovitch spaces $B^{p}_{a p}$.

We also have the following result.

\begin{thm}
	If $f\in L^1(\mathbb{R}, \mathcal{L}(\mathbb{R}), \gamma)$ then
	\begin{equation}
	\label{eq:m1}
	|a(\eta_k, f)|=\left|\int_{\mathbb{R}}f(t)e^{-i\eta_k t} d\gamma(t)\right|\le\norm{f}_{L^1(\mathbb{R}, \mathcal{L}(\mathbb{R}), \gamma)}\norm{e^{-i\eta_k t}}_{L^{\infty}}\le \norm{f}_{L^1(\mathbb{R}, \mathcal{L}(\mathbb{R}), \gamma)}
	\end{equation}
\end{thm}

	Let $\bm{\eta}$ be the counting measure on $2^{\mathbb{R}}$.
	Then the generalized Fourier operator is linear and  maps $L^1(\mathbb{R}, \mathcal{L}(\mathbb{R}), \gamma)$ into $L^{\infty}(\mathbb{R},2^{\mathbb{R}}, \eta) $ and $L^2(\mathbb{R}, \mathcal{L}(\mathbb{R}), \gamma)$ into $L^2(\mathbb{R},2^{\mathbb{R}}, \eta)$ continuously.

	\begin{theorem} Let $1<p<2$ and $0<q<\infty$,  $1/p+1/p'=1$. If $f\in L^{p, q}(\mathbb{R}, \mathcal{L}(\mathbb{R}), \gamma)$, with
		\begin{equation}
		\label{eq:pap2a1}
		a(\eta_k, f)=\int_{\mathbb{R}}f(t)e^{-i\eta_k t} d\gamma(t), \qquad \eta_k \in \mathbb{R}.
		\end{equation}
		Then $\{a(\eta_k, f)\}_{k\in\mathbb{N}}\in L^{p', q}(\mathbb{R},2^{\mathbb{R}}, \eta)$ and 
		\begin{equation}
		\label{eq:m2}
		\norm{\{a(\eta_k, f)\}}_{L^{p', q}(\mathbb{R},2^{\mathbb{R}}, \eta)}\le C(p,q)\norm{f}_{L^{p, q}(\mathbb{R}, \mathcal{L}(\mathbb{R}), \gamma)}.
		\end{equation}
		
	\end{theorem}

\begin{proof}
    By Theorem \ref{thm:cfs1} part (4) the spaces $L^1(\mathbb{R}, \mathcal{L}(\mathbb{R}), \gamma)$, $L^{\infty}(\mathbb{R},2^{\mathbb{R}}, \eta) $, $L^2(\mathbb{R}, \mathcal{L}(\mathbb{R}), \gamma)$, and   $L^2(\mathbb{R},2^{\mathbb{R}}, \eta)$ are  isometrically isomorphic to $L^1(\mathbb{R}', \mathcal{L}'(\mathbb{R}), \gamma')$, $L^{\infty}(\mathbb{R}',2^{\mathbb{R}'}, \eta') $, $L^2(\mathbb{R}', \mathcal{L}'(\mathbb{R}'), \gamma')$, and   $L^2(\mathbb{R}',2^{\mathbb{R}'}, \eta')$ so (\ref{eq:m0}) and (\ref{eq:m1}) holds for the measure spaces. By Marcinkiewicz-Hunt, \cite{Hunt},  (\cite{Grafakos} p. 52), (\ref{eq:m2}) holds for the corresponding measure spaces  and therefore (\ref{eq:m2}) follows.
    
\end{proof}

\begin{remark}
    One may similarly prove the Marcinkiewicz-Hunt theorem for finitely additive measure spaces for $0<p_i, q_i<\infty$. 
\end{remark}

\newpage

\bibliographystyle{abbrv}

\end{document}